\documentclass[12pt]{amsart}
\usepackage{amsmath,amssymb,amsfonts,amsthm,amsopn}
\usepackage{graphics}
\usepackage{srcltx}
\usepackage{tikz}
\newtheorem{tm}{Theorem}[section]
\newtheorem{lemma}[tm]{Lemma}

\newtheorem{theorem}{Theorem}[section]
\newtheorem{corollary}[theorem]{Corollary}

\newtheorem{proposition}[theorem]{Proposition}%%

\newcommand{\beqa}{\begin{eqnarray*}}
\newcommand{\eeqa}{\end{eqnarray*}}

\newcommand{\field}[1]{\mathbb{#1}}
\newcommand{\bR}{\field{R}}        %  real numbers
\newcommand{\bZ}{\field{Z}}        %  whole numbers
\newcommand{\bC}{\field{C}}        %  complex numbers
\def\la{\lambda}

 \def\cF{\mathcal{F}}              % Calligraphic Letters
 \def\cS{\mathcal{S}}
 \def\cD{\mathcal{D}}
 
 \def\cB{\mathcal{B}}

 \def\cM{\mathcal{M}}

 \def\cC{\mathcal{C}}

 \def\cN{\mathcal{N}}
 
 \def\cT{\mathcal{T}}

\def\a{\aleph}

\def\rd{\bR^d}

\def\intrd{\int_{\rd}}

\def\R{\right)}

\def\<{\left<}
\def\>{\right>}

\def\mv1{M_v^1}

\def\o{\xi}
\def\a{\alpha}

\def\z{\zeta}

\def\ZZ{\mathbb{Z}}

\def\N{\mathbb{N}}
\def\R{\mathbb{R}}
\def\Ren{\mathbb{R}^d}
\def\Renn{\mathbb{R}^{2d}}

\def\Fur{\mathcal{F}}

\def\Sn2{S_{2}(L^{2}(\Ren))}
\def\S1{S_{1}(L^{2}(\Ren))}
\def\sig00{\sigma_{0,0}}
\def\la{\langle}
\def\ra{\rangle}

\begin{document}

\title[]{{T}he {C}auchy {P}roblem  for the {V}ibrating {P}late {E}quation in modulation spaces}
%----------Author 1

\author{Elena Cordero and Davide Zucco}
\address{Department of Mathematics,  University of Torino,
Via Carlo Alberto 10, 10123
Torino, Italy}
%\curraddr{}
\email{elena.cordero@unito.it}
\email{davide.zucco@unito.it}
%\thanks{This work was completed with the support of our
%----------Author 2



\date{}
%\dedicatory{To my parents}
%%% ----------------------------------------------------------------------

\begin{abstract} The local solvability of the Cauchy problem for the nonlinear vibrating plate equation
is showed in the framework of modulation spaces.
In the opposite direction, it is proved that there is no local wellposedness in Wiener amalgam spaces even for the solution to the homogeneous
vibrating plate equation.
\end{abstract}

%%% ----------------------------------------------------------------------
\maketitle
%%% ----------------------------------------------------------------------

\section{Introduction and results}
The study of the wellposedness of the Cauchy problem
for the vibrating plate equation (and, more generally, $p$-evolution equations) in
 Sobolev spaces and in Gevrey classes has been  performed extensively by many authors (see, e.g., \cite{AgliardiCicognani,AgliardiZanghirati,Kinoshita} and references therein). The techniques employed there essentially use classical calculus for pseudodifferential operators  and obtain the wellposedness applying a fixed point argument.

In this  note we study the  Cauchy problem for
 the nonlinear vibrating plate equation (NLVP) in the
framework of modulation spaces and Wiener amalgam spaces. Modulation and Wiener amalgam  spaces were introduced by Feichtinger in the 80s \cite{F1,feichtinger80} and soon they revealed
 to be the natural framework for the Time-Frequency Analysis \cite{grochenig}.\par
 Recently, they have been employed in the study of PDE's. In
 particular, let us recall their applications to local wellposedness  for the Schr\"odinger and wave
equation \cite{benyi,benyi3}. Moreover,  let us highlight the deep
and pioneering works \cite{baoxiang2,baoxiang} on the global
wellposedness for nonlinear Schr\"odinger, wave and Klein-Gordon
equations.

The modulation spaces are a family of Banach spaces, which
contains the Sobolev spaces, and can be ``arbitrary" close to the
Schwartz spaces as well as the spaces of tempered distributions. A
Cauchy datum in a modulation space can be
   rougher  than any given one in the standard  fractional Bessel potential setting  (see definition  below).
   Indeed, the roughness of the initial data is useful in many applications.
Contrary to what happens in the case of modulation spaces, we
shall show that there is no wellposedness when we consider initial
data in Wiener amalgam spaces.\par
  Precisely, we shall prove the existence and uniqueness of solutions in modulation spaces to the Cauchy problem for NLVP:
\begin{equation}\label{cpw}
\begin{cases}
\partial^2_t u+\Delta^2_x u=F(u)\\
u(0,x)=u_0(x),\,\,
\partial_t u (0,x)=u_1(x),
\end{cases}
\end{equation}
with $t\in\R$, $x\in\R^d$, $d\geq1$, $\Delta_x=\partial^2_{x_1}+\dots \partial^2_{x_d}$, $\Delta^2 u=\Delta(\Delta u)$. $F$ is a scalar
 function on $\bC$, with
 $F(0)=0$. The solution  $u(t,x)$ is  a complex valued function of $(t,x)\in \R\times\rd$. We will consider
 the case in which $F$ is an entire
 analytic function (in the
 real sense), and we shall highlight the
 special case
 $F(u)=\lambda|u|^{2k}u$,
 $\lambda\in\mathbb{C}$,
 $k\in\mathbb{N}$,
 where we have better
 results.\par
  The arguments mainly rely on recent results for Fourier multipliers. Indeed, the integral version of the
problem \eqref{cpw} has the
form
\begin{equation}\label{solop}
    u(t,\cdot)=K'(t)u_0+K(t)u_1+\mathcal{B}F(u),
\end{equation}
where\begin{equation}\label{op2}
 K'(t)=\cos(t\Delta),\quad
K(t)=\frac{\sin(t\Delta)} {\Delta},\quad
\mathcal{B}=\int _0^t K(t-\tau)\cdot d\tau.
\end{equation}
Here, for every fixed $t$,
the operators $K'(t), K(t)$
in \eqref{op2} are Fourier
multipliers with symbols
\begin{equation}\label{simboli}
\sigma_0(\xi)=\cos(4\pi^2 t|\o|^2),\quad \sigma_1(\xi)=\frac{\sin(4\pi^2 t|\o|^2)}{4\pi^2
|\o|^2},\,\, \o\in\rd.
\end{equation}
We recall
that
 given a function
$\sigma$ on $\rd$ (the
so-called symbol of the
multiplier or,
 simply, multiplier),  the
corresponding Fourier
multiplier operator
$H_{\sigma}$ is formally
defined by
\begin{equation}\label{FM}
H_\sigma f(x)=\intrd e^{2\pi i
x\xi}\sigma(\xi) \hat{f}(\xi)\,d\xi.
\end{equation}
Here the Fourier transform is normalized to be
${\hat
  {f}}(\xi)=\Fur f(\xi)=\int
f(y)e^{-2\pi i t\xi}dy$.
So, continuity properties for
multipliers in suitable
spaces yield estimates for
the linear part of the
equation. These latter are
then combined with classical fixed point
arguments to obtain local
wellposedness in modulation spaces for  \eqref{cpw}.
 \par
   Contrary to what happens for  other equations such as the wave equation \cite[Thm. 4.4]{CNwave},  there is no local wellposedness of \eqref{cpw} in the framework of Wiener amalgam spaces. Precisely, we shall show that the latter already  fails  for  the homogeneous case $F\equiv0$. \par

Although the techniques used are essentially standard, we think it
is worth detailing the study of the Cauchy problem for the NLVP
equation, in view of its many applications  in architecture and
engineering, see for example \cite{KV}.

 In order to state our results,
 we first introduce the spaces we
  deal with
  (\cite{F1,grochenig,benyi3}).
Let $T_x$ and $M_\xi$ be
the so-called translation and
modulation operators, defined by $T_x
g(y)=g(y-x)$ and $M_\xi g(y)=e^{2\pi
i\xi y}g(y)$. Let $g\in\cS(\rd)$ be a
non-zero window function in the Schwartz class  and consider
the so-called short-time Fourier
transform (STFT) $V_gf$ of a
function/tempered distribution $f$ with
respect to the the window $g$:
\[
V_g f(x,\o)=\la f, M_{\o}T_xg\ra =\int e^{-2\pi i \o y}f(y)\overline{g(y-x)}\,dy,
\]
i.e.,  the  Fourier transform $\cF$
applied to $f\overline{T_xg}$. \par For
$s\in\R$, we consider the weight
function $\la x\ra^{s}=(1+|x|^2)^{s/2},
x\in\rd. $ If $1\leq p, q\leq\infty$,
$s\in\R$, the {\it modulation space}
$\mathcal{M}^{p,q}_{s}(\R^d)$ is
defined as the closure of the Schwartz
class $\cS(\rd)$ with respect to the norm
\[
\|f\|_{\mathcal{M}_{s}^{p,q}}=\left(\intrd\left(\intrd |V_gf(x,\o)|^p dx\right)^{q/p}\la\o\ra^{sq} d\o\right)^{1/q}
\]
(with obvious modifications when $p=\infty$ or $q=\infty$). We set
 $\mathcal{M}^p_s=\mathcal{M}^{p,p}_s$,
$\mathcal{M}^{p,q}=\mathcal{M}_{0}^{p,q}$ and
$\mathcal{M}^p=\mathcal{M}^{p}_0$.
\par Modulation spaces are Banach spaces whose  definition is
independent of the choice of the window $g\in\cS(\rd)$. Among
them,  the following well-known function spaces occur:
$\mathcal{M}^{2}(\Ren)=L^2(\Ren)$ and the  Sobolev spaces:
$$ \mathcal{M}^2_{s}(\Ren)=H^s(\Ren)=\{f\,:\,\hat{f}(\o)\la \o\ra^s\in
L^2(\Ren)\}.$$
We also recall the following properties: $\mathcal{M}^{p_1,q_1}_s\hookrightarrow
\mathcal{M}^{p_2,q_2}_s$, if
$p_1\leq p_2$ and $q_1\leq
q_2$,
$(\mathcal{M}_{s}^{p,q})'=\mathcal{M}_{-s}^{p',q'}$.\par
Other properties and more general definitions of modulation spaces can now  be found in  textbooks \cite{grochenig}.

The   Wiener amalgam spaces \cite{feichtinger90} can be defined
using the STFT as well. Namely, for a fixed non-zero window
function $g\in\cS(\rd)$, the Wiener amalgam space $W(\cF
L^p_s,L^q_\gamma)$ is  the closure of the Schwartz class
$\cS(\rd)$ with respect to the norm
\begin{equation}\label{wamalgnorm}
\|f\|_{W({\cF
L}^p_s,L^q_\gamma)}=\left(\int_{\rd}\left(\int_{\rd}|V_{g}f(z,\zeta)|^p\la \zeta\ra^{sp}\,
d\z\right)^{q/p}\, \la z\ra^{\gamma q} dz\right)^{1/q}
\end{equation}
(with obvious modifications when $p=\infty$ or $q=\infty$).
 This definition is independent of the test
function $g\in\cS(\rd)$. \par
For more general definitions and properties of Wiener amalgam spaces we refer to \cite{feichtinger90}.

Observe that, for $p=q,s = 0$ and
 $\gamma= 0$,  we have
\begin{equation}\label{Wienermod}\|f\|_{W({\cF
L}^p,L^p)}=\left(\int_{\rd}\int_{\rd}|V_{g}f(z,\zeta)|^p\,
 dz\,d\z\right)^{1/p}\asymp  \|f\|_{\mathcal{M}^p},
\end{equation}
that is, $W(\cF L^{p},L^p)=\mathcal{M}^p$.

The local wellposedness results for modulation spaces read as follows:
\begin{theorem}\label{T1} Assume $s\geq 0$,
 $1\leq p\leq\infty$,
 $(u_0,u_1)\in \mathcal{M}^{p,1}_s(\rd)\times
\mathcal{M}^{p,1}_{s-2}(\rd)$
and $F(z)=\sum_{j,k=0}^\infty
c_{j,k} z^j \bar{z}^k$, an
entire real-analytic function
on $\bC$ with $F(0)=0$. For
every $R>0$, there exists
$T>0$ such that for every
$(u_0,u_1)$ in the ball $B_R$
of center $0$ and radius $R$
in $\mathcal{M}^{p,1}_s(\rd)
\times\mathcal{M}^{p,1}_{s-2}(\rd)$
there exists a unique
solution $u\in
\cC^0([0,T];\mathcal{M}^{p,1}_s(\rd))$
to \eqref{solop}.
Furthermore, the map
$(u_0,u_1)\mapsto u$ from
$B_R$ to $
\cC^0([0,T];\mathcal{M}^{p,1}_s(\rd))$
is Lipschitz continuous.
\end{theorem}

 For better results concerning the nonlinearity $F(u)=\lambda|u|^{2k} u$ we refer to  Theorem \ref{T2}.

Here the  tools employed follow the pattern of similar Cauchy
problems studied for other equations such as the Schr\"odinger,
wave and Klein-Gordon equations  \cite{benyi,benyi3,CNwave}. Let
us a quote \cite{baoxiang3,baoxiang2,baoxiang} as inspiring works
on this topic. We remark that wave-front properties in the context
of modulation space theory have been achieved for non-linear PDE's
in \cite{pilipovic}.

The norm of $K'(t), K(t)$ in \eqref{op2}, as bounded operators on
modulation spaces, is controlled by the Wiener amalgam norm of the
symbols (see Proposition \ref{L1} below). A new tool for computing
the latter, is given by considering the time variable as a
dilation parameter and applying the dilation properties for
weighted Wiener amalgam spaces contained in \cite{CNK10}.\par The
negative result concerns the unboundedness of the multiplier
$K'(t)$ into the (unweighted) Wiener amalgam spaces $W(\cF
L^p,L^q)$, when $p\not=q$. The case $p=q$ gives $W(\cF
L^{p},L^p)=\mathcal{M}^p$ (see \eqref{Wienermod}) and we come back
to modulation spaces. Indeed, the related Fourier multiplier
$T_\tau$, having symbol $\tau(\xi)=e^{\pi i t|\xi|^2}$, is
unbounded on $W(\cF L^p,L^q)$, when $p\not=q$, see Proposition
\ref{noWiener}.  Hence, in contrast to what happens for other
equations such as the wave equation  \cite[Thm. 4.4]{CNwave},
there is no wellposedness of \eqref{cpw} in Wiener amalgam spaces.

This negative result and  dispersive estimates for the multiplier $K'(t)$, obtained in the study of Schr\"odinger equation \cite{cordero2,cordero3},  suggest to look for  Strichartz estimates in Wiener amalgam spaces. We plan to address this issue in a future work.

\section{Preliminary results and multiplier estimates}
For the local wellposedness  in modulation spaces we need to establish linear and nonlinear estimates
on suitable modulation  spaces that contain the solution $u$.
First of all, we will use estimates for Fourier multipliers on modulation spaces. The following result \cite{CNwave} shows that the multiplier norm, as bounded operator on weighted modulation spaces, can be controlled by a suitable Wiener amalgam  norm of the corresponding symbol.
\begin{proposition} \label{L1} Let $s,t\in\R$, $1\leq p,q\leq\infty$. Let $\sigma$ be a function on $\rd$ and consider the Fourier multiplier operator defined in \eqref{FM}.  If  $\sigma\in W(\cF L^1,
 L^\infty_{t})$, then the operator $H_\sigma$
extends to  a bounded
operator from
$\mathcal{M}^{p,q}_s$
into
$\mathcal{M}^{p,q}_{s+t}$,
with
\begin{equation}\label{fumultest}
    \|H_\sigma f\|_{\mathcal{M}^{p,q}_{s+t}}
    \lesssim \|\sigma\|_{W(\cF L^1,
    L^\infty_{t})}\|f\|_{\mathcal{M}^{p,q}_{s}}.
\end{equation}
\end{proposition}
We shall use the previous criterion to establish the boundedness
of the multipliers $K(t)$ and $K'(t)$. This simply amounts to
looking for the right weighted Wiener norm of the corresponding
symbols. To chase this goal we shall apply the lemmata below (see
\cite{F1} and \cite[Lemma 3.1]{CNwave}, respectively for their
proofs).
\begin{lemma}\label{propWiener}
  For $i=1,2,3$, let $B_i$ be one of the Banach spaces
  $\cF L^q_s$ ($1\leq q\leq\infty$,
   $s\in\R$),  $C_i$ be one of the Banach spaces
  $L^p_\gamma$ ($1\leq p\leq\infty$, $\gamma\in\R$).
  If $B_1\cdot B_2\hookrightarrow B_3$ and $C_1\cdot
  C_2\hookrightarrow C_3$, we have
  \begin{equation}\label{point0}
  W(B_1,C_1)\cdot W(B_2,C_2)\hookrightarrow W(B_3,C_3).
  \end{equation}
\end{lemma}

\begin{lemma}\label{L3}
Let  $R>0$ and
$f\in\cC_0^\infty(\rd)$ such
that supp $f\subset
B(y,R):=\{x\in\rd, |x-y|\leq
R\}$, with $y\in\rd$. Then,
for every $0<p\leq \infty$,
there exist an index
$k=k(p)\in \N$ and a constant
$C_{R,p}>0$ (which depends
only on $R$ and $p$)  such
that
\begin{equation}\label{E1}
                \|f\|_{\cF L^p}\leq C_{R,p}
                \sup_{|\a|\leq k} \|\partial^\alpha f\|_{L^\infty}.
\end{equation}
\end{lemma}

%
%In this section we first
%prove estimates in
%$M^{p,q}_s$ for the
%multipliers $K(t), K'(t)$ defined in \eqref{op2}.

Let us now introduce the symbols
\begin{equation}\label{multsymb}
    \tilde{\sigma}_0(\xi)= \cos|\xi|^2\quad\mbox{and }\quad \tilde{\sigma}_1(\xi)= \frac{\sin |\o|^2}{|\xi|^2}.
\end{equation}
The boundedness of the  Fourier multipliers having symbols \eqref{multsymb} is a consequence of the issue below.
\begin{proposition}\label{A1} (i) The multiplier
$\tilde{\sigma}_{0}$ in
\eqref{multsymb} is in the
space $W(\cF
L^1,L^\infty)$.
(ii) The multiplier $\tilde{\sigma}_1$
in \eqref{multsymb} is in
$W(\cF L^1,L^\infty_2)$.
\end{proposition}
\begin{proof}   ({\it i}). Since $\tilde{\sigma}_{0}(\xi)=\frac{e^{i |\xi|^2}+ e^{-i|\xi|^2}}{2}$, the result follows from \cite[Theorem 9]{benyi}. \\
\noindent
({\it ii}). The proof follows the pattern  of \cite[Proposition 3.1]{CNwave}. Consider a function $\chi\in\cC^\infty_0(\rd)$, $ 1\leq \chi(\xi)\leq
2$, such that $\chi(\o)=1$ if
$|\o|\leq 1$, whereas
 $\chi(\o)=0$ if $|\o|\geq 2$. Then,
\begin{equation}\label{separazione}
\tilde{\sigma}_1=\chi(\o)\tilde{\sigma}_1(\o)+(1-\chi(\o))\tilde{\sigma}_1(\o):=\sigma_{sing}(\xi)+\sigma_{osc}(\xi).
\end{equation}
{\em Singularity at the
origin}. Since $\sigma_{sing}\in \cC_0^{\infty}(\rd)\subset W(\cF
L^1,L^\infty_s)$, for every $s\in\R$, the claim is proved.
\par
\noindent {\em Oscillation at
infinity.} We can split $\sigma_{osc}$ into
$$\sigma_{osc}(\xi)=\sin|\xi|^2\cdot \frac{1-\chi(\xi)}{|\xi|^2}.
$$
the first term $\sin|\xi|^2$ is in $W(\cF L^1,L^\infty)$ (see {\it (i)}), hence, if we show that the multiplier $\frac{1-\chi(\xi)}{|\xi|^2}$  is in  $W(\cF
L^p,L^\infty_2)$, the pointwise multiplication properties for Wiener amalgam spaces (cf. Lemma \ref{propWiener}) give the claim.  To prove the latter inclusion we use  Lemma \ref{L3}, applied to the function
  $\frac{1-\chi(\xi)}{|\xi|^2} T_x g\in\cC^\infty_0(\rd)$, with $g\in\cC^\infty_0(\rd)$.
  Precisely,
  $$\left|\partial^\a \left(\frac{1-\chi(\xi)}{|\xi|^2}\right)\right|\lesssim
  \la\o\ra^{-2},\quad |\partial^\a g(\o-x)|\lesssim
\la x-\o\ra^{-N},\quad
 \forall x,\o \in\rd,\ \forall N\in\mathbb{N},\, \forall \a\in\bZ^d_+.$$
Combining the preceding
estimates with the weight property  $\la\o\ra^{-\delta}\la x-\o\ra^{-|\delta|}
 \leq \la x\ra^{-\delta}$, we conclude the proof.
\end{proof}

\noindent
\begin{corollary}\label{C1}
Let $s\in\R$, $j=0,1$. For every $1\leq
p, q\leq\infty$, the Fourier
multiplier
$H_{\tilde{\sigma}_{j}}$,
with symbol
$\tilde{\sigma}_{j}$ defined
in \eqref{multsymb}, extends
to a bounded operator from
$\mathcal{M}^{p,q}_s(\rd)$
into
$\mathcal{M}^{p,q}_{s+2j}(\rd)$,
with
\begin{equation}\label{fumultest3}
    \|H_{\tilde{\sigma}_{j}} f\|_{\mathcal{M}^{p,q}_{s+2j}}\lesssim \|{\tilde{\sigma}_{j}}\|_{W(\cF L^1,
    L^\infty_{2j})}\|f\|_{\mathcal{M}^{p,q}_{s}}.
\end{equation}
\end{corollary}
\begin{proof}
The desired result follows
from Propositions \ref{A1}
and  \ref{L1}.
\end{proof}

We remark that the boundedness result for $H_{\tilde{\sigma}_{0}}$
was first proved in \cite{toft}.

We first shall show  the local wellposedness in modulation spaces of the  Cauchy problem for the homogeneous vibrating plate equation:
\begin{equation}\label{CPH}
\begin{cases}
\partial^2_t u+\Delta^2_x u=0\\
u(0,x)=u_0(x),\,\,
\partial_t u (0,x)=u_1(x).\,\,
\end{cases}
\end{equation}
This is obtained by means of  the previous estimates, combined with dilation properties for weighted Wiener amalgam spaces.

For $1\leq p\leq\infty$, let $p'$ be the conjugate exponent of $p$
($1/p+1/p'=1$). For $(1/p,1/q)\in [0,1]\times [0,1]$, using the
 notation introduced in \cite{sugimototomita}, we define the
subsets
$$ I_1=\max (1/p,1/p')\leq 1/q,\quad\quad I_1^*=\min (1/p,1/p')\geq 1/q,
$$
$$ I_2=\max (1/q,1/2)\leq 1/p',\quad\quad I_2^*=\min (1/q,1/2)\geq  1/p',
$$
$$ I_3=\max (1/q,1/2)\leq 1/p,\quad\quad I_3^*=\min (1/q,1/2)\geq
1/p,
$$
as shown in Figure 1:
\begin{figure}
\centering
\begin{minipage}[t]{.40\textwidth}
\centering
 \begin{tikzpicture}[>=latex,scale=1.2]
  \coordinate (a) at (0,0) {}; 
  \coordinate (b) at (4,0) {}; 
  \coordinate (c) at (0,4) {} ; 
  \coordinate (d) at (1.5,0) {} ;
  \node (dn) at (1.5,0) [below]{$1/2$} ; 
  \coordinate (e) at (0,1.5) {} ; 
  \node (en) at (0,1.5) [left]{$1/2$} ; 
  \coordinate (f) at (0,3) {} ; 
  \node (fn) at (0,3) [left]{$1$};
  \coordinate (g) at (3,0) {} ; 
  \node (gn) at (3,0) [below]{$1$};
  \coordinate (h) at (3,3) {} ;
  \coordinate (i) at (1.5,1.5) {} ;
  \node (in1) at (1.5,2.25)  {$I_1$};
  \node (in2) at (0.75,0.75) {$I_2$};
  \node (in3) at (2.25,0.75) {$I_3$};

  \draw[->,thick] (a)node[below]{$0$} -- (b) node[below]{$1/p$};
  \draw[->,thick] (a)--(c)node[left]{$1/q$};
  \draw[-,thick] (f) -- (h);
  \draw[-,thick] (g) -- (h);
  \draw[-,thick] (f) -- (i);
  \draw[-,thick] (d) -- (i);
  \draw[-,thick] (h) -- (i);

  \node (didascalia) at (1.5,-1) {$0<\lambda\leq 1$};
 \end{tikzpicture}
 \end{minipage}
 \hspace*{3mm}
 \begin{minipage}[t]{.40\textwidth}
 \centering

 \begin{tikzpicture}[>=latex,scale=1.2]
  \coordinate (a) at (0,0) {}; 
  \coordinate (b) at (4,0) {}; 
  \coordinate (c) at (0,4) {} ; 
  \coordinate (d) at (1.5,3) {} ;
  \node (dn) at (1.5,0) [below]{$1/2$} ; 
  \coordinate (e) at (0,1.5) {} ; 
  \node (en) at (0,1.5) [left]{$1/2$} ; 
  \coordinate (f) at (0,3) {} ; 
  \node (fn) at (0,3) [left]{$1$};
  \coordinate (g) at (3,0) {} ; 
  \node (gn) at (3,0) [below]{$1$};
  \coordinate (h) at (3,3) {} ;
  \coordinate (i) at (1.5,1.5) {} ;
  \node (in1) at (1.5,0.75)  {$I_1^*$};
  \node (in2) at (0.75,2.25) {$I_3^*$};
  \node (in3) at (2.25,2.25) {$I_2^*$};

  \draw[->,thick] (a)node[below]{$0$} -- (b) node[below]{$1/p$};
  \draw[->,thick] (a)--(c)node[left]{$1/q$};
  \draw[-,thick] (f) -- (h);
  \draw[-,thick] (g) -- (h);
  \draw[-,thick] (a) -- (i);
  \draw[-,thick] (g) -- (i);
  \draw[-,thick] (d) -- (i);

  \node (didascalia) at (1.5,-1) {$\lambda\geq 1$};
 \end{tikzpicture}
\end{minipage}
\caption{The index sets.}
\end{figure}
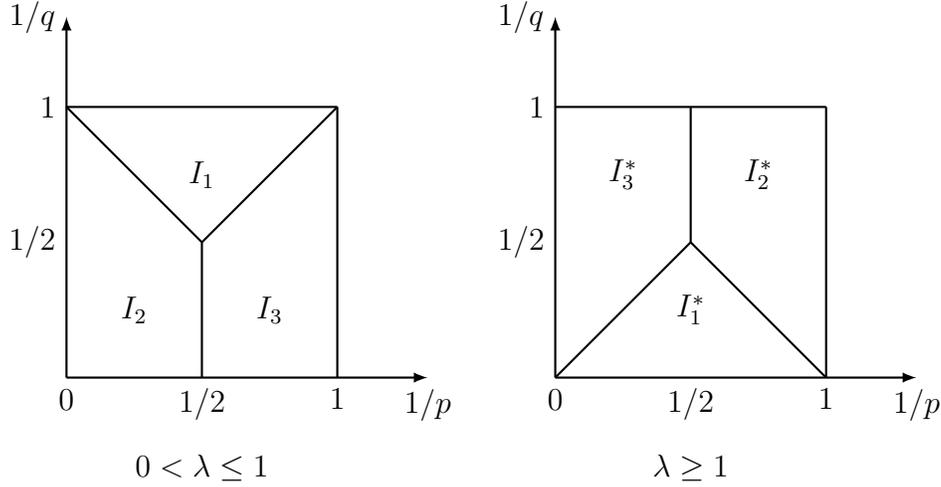

\vspace{1.5cm}
We introduce the indices:
$$ \mu_1(p,q)=\begin{cases}-1/p &  \quad {\mbox{if}} \quad (1/p,1/q)\in  I_1^*,\\
 1/q-1 &   \quad {\mbox{if}}  \quad (1/p,1/q)\in  I_2^*,\\
 -2/p +1/q&  \quad  {\mbox{if}}  \quad (1/p,1/q)\in  I_3^*,\\
 \end{cases}
 $$
and
$$ \mu_2(p,q)=\begin{cases}-1/p &  \quad {\mbox{if}} \quad (1/p,1/q)\in  I_1,\\
 1/q-1 &   \quad {\mbox{if}}  \quad (1/p,1/q)\in  I_2,\\
 -2/p +1/q&  \quad  {\mbox{if}}  \quad (1/p,1/q)\in  I_3.\\
 \end{cases}
 $$
Dilation properties for un-weighted modulation spaces have been
completely  studied in \cite{sugimototomita}. The dilation
properties for  weighted modulation spaces and Wiener amalgam
spaces have  recently been developed in \cite{CNK10}. In
particular, let us recall the following result:
 \begin{proposition}\label{mainbothW}
Let $1\leq p,q \leq\infty$, $t,s\in\R$. Then the following are true:\\
(1) There exists a constant $C>0$ such that $\forall f \in W(\cF L^p_s,L^q_t),\,\lambda\geq 1,$
\begin{align*}
C^{-1} \lambda^{d\mu_2(p',q')}\min\{1,\lambda^{t}\}&\min\{1,\lambda^{-s}\}\,\|f\|_{W(\cF L^p_s,L^q_t)} \leq \| f_\lambda\|_{W(\cF L^p_s,L^q_t)}\\[1 \jot]
&  \leq  C \lambda^{d\mu_1(p',q')}\max\{1,\lambda^{t}\}\max\{1,\lambda^{-s}\}\,\|f\|_{W(\cF L^p_s,L^q_t)}.
\end{align*}
 (2) There exists a constant $C>0$ such that $\forall f \in W(\cF L^p_s,L^q_t),\,0\leq\lambda\leq 1,$
\begin{align*}
C^{-1} \lambda^{d\mu_1(p',q')}\min\{1,\lambda^{t}\}&\min\{1,\lambda^{-s}\}\,\|f\|_{W(\cF L^p_s,L^q_t)} \leq \| f_\lambda\|_{W(\cF L^p_s,L^q_t)}\\[1 \jot]
&   \leq C\lambda^{d\mu_2(p',q')}\max\{1,\lambda^{t}\}\max\{1,\lambda^{-s}\}\,\|f\|_{W(\cF L^p_s,L^q_t)}.
\end{align*}
\end{proposition}

The symbols $\sigma_0, \sigma_1$ in \eqref{op2} can be rewritten as time dilations of the symbols  $\tilde{\sigma}_0\in W(\cF L^1, L^\infty)$, $\tilde{\sigma}_1\in W(\cF L^1,L^\infty_{2})$ in \eqref{multsymb}. Precisely,  for $t>0$, we can write $\sigma_0(\xi)=(\tilde{\sigma}_0)_{2\pi \sqrt{t}}$, $\sigma_1(\xi)=t( \tilde{\sigma}_1)_{2\pi \sqrt{t}}$.
Using Proposition \ref{mainbothW} with $\mu_1(\infty,1)=1$, $\mu_2(\infty,1)=0$, we have, for every $R>0$,
\begin{equation*}
\|(\tilde{\sigma}_0)_{2\pi \sqrt{t}}\|_{W(\cF L^1,
    L^\infty)}\leq   \begin{cases} C_{0,R} \|\tilde{\sigma}_0\|_{W(\cF L^1,
    L^\infty)} ,\quad t\leq R\\
 C'_{0,R}  t^{\frac d 2}\|\tilde{\sigma}_0\|_{W(\cF L^1,
    L^\infty)}, \quad t\geq R.\end{cases}
 \end{equation*}
and
\begin{equation*}
\|(\tilde{\sigma}_1)_{{2\pi \sqrt{t}}}\|_{W(\cF L^1,
    L^\infty_{2})}\leq   \begin{cases} C_{1,R} \|\tilde{\sigma}_1\|_{W(\cF L^1,
    L^\infty_{2})} ,\quad t\leq R\\
 C'_{1,R} t^{\frac d2+1}\|\tilde{\sigma}_1\|_{W(\cF L^1,
    L^\infty_{2})}, \quad t\geq R.\end{cases}
 \end{equation*}
Hence the Cauchy problem \eqref{CPH} admits  a solution  $u(t,x)$ satisfying:
\begin{equation}\label{homesol}
\|u(t,\cdot)\|_{\mathcal{M}^{p,q}_{s}}\leq C_0 (1+ t)^{\frac d 2}\|u_0\|_{\mathcal{M}^{p,q}_{s}}+ C_1 t(1+t)^{\frac d 2+1} \|u_1\|_{\mathcal{M}^{p,q}_{s-2}}, \quad t> 0,
\end{equation}
for every $1\leq p, q \leq \infty$, $s\in \R$.
\vskip0.1truecm

Contrary to what happens for other equations such as the wave equation (see \cite[Thm. 4.4]{CNwave}), there is no wellposedness of \eqref{cpw} or \eqref{CPH} into the Wiener amalgam spaces $W(\cF L^p_s,L^q_\gamma)$.
The reason being the unboundedness of the multiplier $K'(t)$ into the  Wiener amalgam spaces $W(\cF L^p,L^q)$, when $p\not=q$; the case $p=q$ gives $W(\cF L^p,L^p)=\mathcal{M}^p$ (see \eqref{Wienermod}). In this case the boundedness of $K'(t)$ was first proved in \cite[Theorem 1]{benyi}.

To prove the unboundedness of  $K'(t)$, we first recall  the relationship between modulation and Wiener amalgam spaces when $p\not=q$ \cite{feichtinger90}:
\begin{proposition}\label{vm} The
Fourier transform establishes
an isomorphism  $\Fur:
\mathcal{M}^{p,q}\to W(\Fur
L^p,L^q)$.\par
\end{proposition}
The crucial tool concerns the unboundedness  of the pointwise multiplication for modulation spaces \cite[Proposition 7.1]{fio1}:
\begin{proposition}\label{contro}
The multiplication
$U_{I_{d}} f= e^{\pi i |\xi|^2}f$, $f\in\cS(\rd)$, is unbounded on
$\mathcal{M}^{p,q}$, for every $1\leq
p,q\leq\infty$, with
$p\not=q$.
\end{proposition}
Now, we are ready to prove the unboundedness of the multiplier $K'(t)$. This immediately follows by  the unboundedness of the multiplier $T_\tau$ below.
\begin{proposition}\label{noWiener} The Fourier multiplier $T_\tau$, having symbol $\tau(\xi)=e^{\pi i t|\xi|^2}$, is
unbounded on every $W(\cF L^p,L^q)$,$1\leq p,q\leq\infty$, with $p\not=q$.
\end{proposition}
\begin{proof}
For every $f\in \cS(\rd)$, using Proposition \ref{vm},  we have
\begin{equation*}
\|T_\tau f\|_{W(\cF L^p,L^q)}\asymp\|\cF(T_\tau f)\|_{\mathcal{M}^{p,q}}=\|\tau\hat{f}\|_{\mathcal{M}^{p,q}}.
\end{equation*}
Hence $T_\tau$ is bounded on  $W(\cF L^p,L^q)$ if and only if  the multiplication
$U_{I_{d}}$ is bounded on
$\mathcal{M}^{p,q}$, and this happens if and only if $p=q$, thanks to Proposition \ref{contro} (necessary conditions) and \cite[Theorem 1]{benyi} (sufficient conditions).
\end{proof}

\section{Local  wellposedness of NLVP on modulation spaces}
In this section we present the wellposedness
result on modulation spaces. To establish  nonlinear estimates
on appropriate modulation spaces we shall use the lemma below. It was first proved in
\cite{F1} (see also \cite[Corollary 4.2]{baoxiang}).
\begin{lemma} \label{L2} Let $s\geq 0$,
 $1\leq p\leq p_i\leq\infty$, $1\leq r,q_i\leq\infty$, $N\in\N$, satisfy
\begin{equation}\label{indices}
\sum_{i=1}^N \frac
1{p_i}=\frac1p,\quad
\sum_{i=1}^N \frac
1{q_i}=N-1+\frac1r,
                \end{equation}
                then we have
$$
\| \prod_{i=1}^N
u_i\|_{\mathcal{M}^{p,r}_s}\leq
\prod_{i=1}^N \|
u_i\|_{\mathcal{M}^{p_i,q_i}_s}.
$$
                \end{lemma}
 In particular, for
$p_i=N p$, $q_i=q$,
$i=1,\dots N$, we get
\begin{equation}\label{norms}
\|\prod_{i=1}^N
u_i\|_{\mathcal{M}^{p,r}_s}\leq
\prod_{i=1}^N\|
u_i\|_{\mathcal{M}^{p,q}_s},\quad
\frac{N}{q}=N-1+\frac1r.
\end{equation}
The proof of the local existence theory uses  the following variant of the contraction mapping theorem (see, e.g.,
\cite[Proposition
1.38]{tao}).
\begin{proposition}\label{AIA}
Let $\cN$ and $\cT$ be two
Banach spaces. Suppose we are
given a linear operator
$\cB:\cN\to \cT$ with the
bound
\begin{equation}\label{aia1}
\|\cB f\|_{\cT}\leq
C_0\|f\|_{\cN}
\end{equation}
for all $f\in\cN$ and some
$C_0>0$, and suppose that we
are given a nonlinear
operator $F:\cT\to\cN$ with
$F(0)=0$, which obeys the
Lipschitz bounds
\begin{equation}\label{aia2}
\|F(u)-F(v)\|_{\cN}\leq\frac{1}{2C_0}\|u-v\|_{\cT}
\end{equation}
for all $u,v$ in the ball
$B_\mu:=\{u\in\cT:
\|u\|_\cT\leq \mu \}$, for
some $\mu>0$. Then, for all
$u_{\rm lin}\in B_{\mu/2}$
there exists a unique
solution $u\in B_\mu$ to the
equation
\[
u=u_{\rm lin}+\cB F(u),
\]
with the map $u_{lin}\mapsto
u$ Lipschitz with constant at
most $2$ (in particular,
$\|u\|_{\cT}\leq 2\|u_{\rm
lin}\|_{\cT}$).
\end{proposition}
We are now ready  to prove Theorem \ref{T1}.
\begin{proof}[Proof of Theorem \ref{T1}]
We first observe that, by  \eqref{homesol}, for every
$1\leq p\leq\infty$, the
multiplier $K'(t)$, in \eqref{op2} can
be extended to a bounded
operator on
$\mathcal{M}^{p,1}_s$, with
\begin{equation}\label{G1} \|K'(t)
u_0\|_{\mathcal{M}^{p,1}_s}\leq
C_0 (1+ t)^{\frac d 2} \|u_0\|_{\mathcal{M}^{p,1}_s},\quad
t>0.
\end{equation}
Similarly, the multiplier
operator $K(t)$
   satisfies the
  estimate
\begin{equation}\label{G2}
\|K(t)
u_1\|_{\mathcal{M}^{p,1}_s}\leq
C_1 t(1+t)^{\frac d 2+1}\|u_1\|_{\mathcal{M}^{p,1}_{s-2}},\
t>0,\end{equation} for
every $1\leq p\leq \infty$.\par
 Now we
are going to apply Proposition \ref{AIA} with
$\cT=\cN=C^0([0,T];\mathcal{M}^{p,1}_s)$, where  $T>0$ will be
chosen  later on, with the nonlinear operator $\cB$ given by the
Duhamel operator in \eqref{op2}. Here $u_{\rm
lin}:=K'(t)u_0+K(t)u_1$ is in the ball $B_{\mu/2}\subset\cT$ by
\eqref{G1}, \eqref{G2}, if $\mu$ is sufficiently large, depending
on $R$. Using Minkowski integral inequality and \eqref{G2}, we
obtain \eqref{aia1}. Namely,
$$\|\cB u\|_{\mathcal{M}^{p,1}_s}\leq T C_T \|u\|_{_{\mathcal{M}^{p,1}_{s-2}}}\leq T C_T \|u\|_{_{\mathcal{M}^{p,1}_{s}}},
$$
with $C_T=C_1 \sup_{t\in[0,T]}t(1+t)^{\frac d 2+1}$ and using the inclusion $\mathcal{M}^{p,1}_s\hookrightarrow
\mathcal{M}^{p,1}_{s-2}$.
\par Condition \eqref{aia2} is already proved in \cite[Theorem 4.1]{CNwave}. There, applying the relation
\eqref{norms} for $q=r=1$, the following estimate is obtained:
$$\|F(u)-F(v)\|_{\cM^{p,1}_s}\leq
 \|u-v\|_{\cM^{p,1}_s}\sum_{j,k,l,m\geq0}(|c_{j,k,l,m}|+
 |c'_{j,k,l,m}|)
 \|u\|_{\mathcal{\cM}^{p,1}_{s}}^{j+k}
\|v\|_{\mathcal{\cM}^{p,1}_{s}}^{l+m}
<\infty,$$
for $u,v\in
\cM^{p,1}_s$.  This expression is  $\leq C_\mu \|u-v\|_{\cM^{p,1}_s}$
 if $u,v\in B_\mu$.
Hence, by choosing $T$
sufficiently small we
conclude the proof of
existence, and also that of
uniqueness among the solution
in $\cT$ with norm $O(R)$.
Finally, this last constraint can be
eliminated by a standard
continuity argument (cf. the
proof of Proposition 3.8 in
\cite{tao}).
\end{proof}

A better result can be obtained when considering  the nonlinearity
\begin{equation} \label{PW}
F(u)=F_k(u)=\lambda
|u|^{2k}u=\lambda
u^{k+1}\bar{u}^k, \quad
\lambda\in\mathbb{C},\
k\in\N.
\end{equation}

\begin{theorem}\label{T2}
Let $F(u)$ be as in \eqref{PW},
 $1\leq p\leq \infty$, $s\geq2$, and
 \begin{equation}\label{indr}
                q'>kd.
\end{equation}
For every $R$ there exists
$T>0$ such that for every
$(u_0,u_1)$ in the ball $B_R$
of center $0$ and radius $R$
in $\mathcal{M}^{p,q}_s(\rd)
\times\mathcal{M}^{p,q}_{s-2}(\rd)$
there exists a unique
solution $u\in
\cC^0([0,T];\mathcal{M}^{p,q}_s(\rd))$
to \eqref{solop}. Furthermore
the map $(u_0,u_1)\mapsto u$
from $B_R$ to $
\cC^0([0,T];\mathcal{M}^{p,q}_s(\rd))$
is Lipschitz continuous.
\end{theorem}
\begin{proof} Here we set $\cT=
\cC^0([0,T];\mathcal{M}^{p,q}_s(\rd))$,
$\cN=
\cC^0([0,T];\mathcal{M}^{p,q}_{s-2}(\rd))$.
 Now
$F_k(z)-F_k(w)=(z-w)p_k(z,w)+(\overline{z}-\overline{w})
q_k(z,w)$, where $p_k,q_k$
are polynomials of degree
$2k$ in
$z,w,\overline{z},\overline{w}$
($q_0(z,w)\equiv0$). Using
\eqref{norms} for
$1\leq p\leq\infty$,
 we obtain
$$
\|F(u)-F(v)\|_{\mathcal{M}^{p,r}_{s-2}}\leq
C
|\lambda|\|u-v\|_{\mathcal{M}^{p,q}_{s-2}}
(\|u\|_{\mathcal{M}^{p,q}_{s-2}}^{2k}+
\|v\|_{\mathcal{M}^{p,q}_{s-2}}^{2k}),$$
with
\begin{equation}\label{esr}
r=\frac q{2k(1-q)+1}.
\end{equation}
\noindent The inclusion
relations for modulation
spaces
 \cite{F1,baoxiang2} fulfill
$$\mathcal{M}^{p,r}_s\hookrightarrow
\mathcal{M}^{p,q}_{s-2}
\quad\mbox{if}\quad
\frac{d}{q}-\frac{d}{r}<2,
$$
which, combined with \eqref{esr}, yields  \eqref{indr}. So
\eqref{aia2} is verified and
we are done.
\end{proof}

We observe that condition \eqref{indr} let us consider initial data $u_0,u_1$ in rougher spaces than those in the Cauchy problem for  the wave equation \cite[Theorem 4.3]{CNwave}.

\section*{Acknowledgements}
The authors would like to thank Professors Luigi Rodino and Fabio
Nicola for fruitful conversations and comments. We are grateful to
the anonymous referee for his valuable comments.

\end{document}